\documentclass[final]{amsart}
\usepackage{amssymb}
\usepackage{array}
\usepackage{amsthm}
\usepackage[all]{xy}

\newtheorem{theorem}{Theorem}[section]

\newtheorem{proposition}[theorem]{Proposition}
\newtheorem{corollary}[theorem]{Corollary}
\newtheorem{claim}[theorem]{Claim}
\newtheorem{problem}[theorem]{Problem}
\newtheorem{conjecture}[theorem]{Conjecture}
\theoremstyle{definition}

\newtheorem{remark}[theorem]{Remark}
\newtheorem{question}[theorem]{Question}

\newcommand{\PPC}{\mathsf{PPC}}
\newcommand{\IN}{\mathbb N}
\newcommand{\e}{\varepsilon}
\newcommand{\IZ}{\mathbb Z}
\newcommand{\IC}{\mathbb C}

\newcommand{\IT}{\mathbb T}
\newcommand{\IR}{\mathbb R}

\newcommand{\Cyclic}{\mathsf{Cyclic}}

\newcommand{\Boolean}{\mathsf{Boolean}}

\newcommand{\Dihedral}{\mathsf{Dihedral}}

\title{Difference bases in cyclic groups}
\author{Taras Banakh and Volodymyr Gavrylkiv}

\address[T.~Banakh]{Ivan Franko National University of Lviv (Ukraine),
and \newline Institute of Mathematics, Jan Kochanowski University in
Kielce (Poland)}
\email{t.o.banakh@gmail.com}
\address[V.~Gavrylkiv]{Vasyl Stefanyk Precarpathian National
University,
Ivano-Frankivsk, Ukraine} \email{vgavrylkiv@gmail.com}

\begin{document}

\begin{abstract} A subset $B$ of an Abelian group $G$ is called a {\em
difference basis} of $G$ if each element $g\in G$ can be written as the
difference $g=a-b$ of some elements $a,b\in B$. The smallest
cardinality $|B|$ of a difference basis $B\subset G$ is called the {\em
difference size} of $G$ and is denoted by $\Delta[G]$.
We prove that for every $n\in\IN$ the cyclic group $C_n$ of order $n$ has difference
size $\frac{1+\sqrt{4|n|-3}}2\le \Delta[C_n]\le\frac32\sqrt{n}$. If $n\ge 9$ (and $n\ge 2\cdot 10^{15}$), then $\Delta[C_n]\le\frac{12}{\sqrt{73}}\sqrt{n}$ (and  $\Delta[C_n]<\frac2{\sqrt{3}}\sqrt{n}$).
Also we calculate the difference sizes of all cyclic groups of
cardinality $\le 100$.
\end{abstract}
\maketitle

\section{Introduction}

A subset $B$ of a group $G$ is called a {\em difference basis} for a subset $A\subset G$ if each element $a\in A$ can be written as $a=xy^{-1}$ for some $x,y\in B$. If the group operation of $G$ is denoted by $+$, then the element $xy^{-1}$ is written as the difference $x-y$ (which justifies the choice of the terminology).

 The smallest cardinality of a difference basis for $A\subset G$ is called the {\em difference size} of $A$ and is denoted by $\Delta[A]$. For example, the set $\{0,1,4,6\}$ is a difference basis for the interval $A=[-6,6]\cap\IZ$ witnessing that $\Delta[A]\le 4$. In Proposition~\ref{p:BGN}(4) we shall prove that the difference size is subadditive in the sense that $\Delta[A\cup B]<\Delta[A]+\Delta[B]$ for any non-empty subsets $A,B$ of a group $G$.

The definition of a difference basis $B$ for a set $A$ in a group $G$ implies that $|A|\le |B|^2$ and hence $\Delta[A]\ge \sqrt{|A|}$. The fraction
$$\eth[A]:=\frac{\Delta[A]}{\sqrt{|A|}}\ge1$$is called the {\em difference characteristic} of $A$. The difference characteristic is submultiplicative in the sense that $\eth[G]\le \eth[H]\cdot\eth[G/H]$ for any normal subgroup $H$ of a finite group $G$, see \cite[1.1]{BGN}.

In this paper we are interested in evaluating the difference characteristics of finite cyclic groups. In fact, this problem has been studied in the literature (see \cite{BH}, \cite{FKL}, \cite{KL}). In particular, Kozma and Lev \cite{KL} proved (using the classification of finite simple groups) that each finite group $G$ has difference characteristic $\eth[G]\le\frac{4}{\sqrt{3}}\approx 2.3094$. In this paper we shall show that for finite cyclic groups this upper bound can be improved to $\eth[G]\le\frac32$. Moreover, if a finite cyclic group $C_n$ has cardinality $n\ge 9$ (resp. $n\ge2\cdot 10^{15}$), then $\eth[C_n]\le\frac{12}{\sqrt{73}}\approx1.4045$ (resp.  $\Delta[C_n]<\frac2{\sqrt{3}}\approx1.1547$). It is an open problem if $\lim_{n\to\infty}\eth[C_n]=1$. However, many subsequences of the sequence $(\eth[C_n])_{n=1}^\infty$ indeed converge to $1$. In particular, using known results on (relative) difference sets, we shall prove that $$\lim_{p\to\infty}\eth[C_{p^2-p}]=\lim_{q\to \infty}\eth[C_{q^2+q+1}]=\lim_{q\to\infty}\eth[C_{q^2-q}]=1$$ where $p$ runs over prime numbers and $q$ runs over prime powers.
A number $q$ is called a  {\em prime power} if $q$ is equal to the power $p^k$ of some prime number $p$.

To derive an upper bound for the difference sizes of arbitrary finite cyclic group, we shall use known information on the difference sizes of order intervals $[-n,n]=\{x\in\IZ:|x|\le n\}$ in the group $\IZ$ of integer numbers. Here we exploit the approach first used by R\'edei and R\'enyi \cite{RR} and then developed by Leech \cite{Leech} and Golay \cite{Golay}.

For a model of a cyclic group of order $n$ we take the multiplicative group $$C_n=\{z\in\IC:z^n=1\}$$ of complex $n$-th roots of $1$. 

\section{Known results}

In this section we recall some known results on difference bases in
finite groups. The following important fact was proved by Kozma and Lev
\cite{KL} (using the classification of finite simple groups).

\begin{theorem}[Kozma, Lev]\label{KL} Each finite group $G$ has
difference characteristic $\eth[G]\le\frac{4}{\sqrt{3}}$.
\end{theorem}

For cyclic groups the upper bound $\frac{4}{\sqrt{3}}$ can be improved to $\frac32$, which will be done in Theorem~\ref{t:max}.

For a real number
 $x$ we put $$\lceil x\rceil=\min\{n\in\IZ:n\ge x\}\mbox{ and }\lfloor
x\rfloor=\max\{n\in\IZ:n\le x\}.$$

The first three statements of the following proposition were proved in \cite[1.1]{BGN}.

\begin{proposition}\label{p:BGN} Let $G$
be a finite group. Then
\begin{enumerate}
\item $ \frac{1+\sqrt{4|G|-3}}2\le \Delta[G]\le
\big\lceil\frac{|G|+1}2\big\rceil$,
\item $\Delta[G]\le \Delta[H]\cdot \Delta[G/H]$ and
$\eth[G]\le\eth[H]\cdot\eth[G/H]$ for any normal subgroup
$H\subset G$;
\item $\Delta[G]\le |H|+|G/H|-1$ for any subgroup $H\subset G$;
\item $\Delta[A\cup B]\le\Delta[A]+\Delta[B]-1$ for any non-empty sets $A,B\subset G$.
\end{enumerate}
\end{proposition}

\begin{proof} Since (1)--(3) are proved in \cite[1.1]{BGN}, we shall explain (4). Given non-empty sets $A,B\subset G$, find difference bases $D_A$ and $D_B$ for the sets $A,B$ of cardinality $|D_A|=\Delta[A]$ and $|D_B|=\Delta[B]$. Taking any point $d\in D_A$ and replacing $D_A$ by its shift $D_Ad^{-1}$, we can assume that the unit $1_G$ of the group $G$ belongs to $D_A$. By the same reason, we can assume that $1_G\in D_B$. The union $D=D_A\cup D_B$ is a difference basis for $A\cup B$, witnessing that $$\Delta[A\cup B]\le |D|\le |D_A|+|D_B|-1=\Delta[A]+\Delta[B]-1.$$
\end{proof}

Finite groups $G$ with $\Delta[G]=\big\lceil\frac{|G|+1}2\big\rceil$
were characterized in \cite{BGN} as follows.

\begin{theorem}[Banakh, Gavrylkiv, Nykyforchyn] \label{upbound}
For a finite group $G$
\begin{enumerate}
\item[(i)] $\Delta[G]=\big\lceil\frac{|G|+1}2\big\rceil>\frac{|G|}2$
if and only if $G$ is isomorphic to one of the groups:\newline
$C_1$, $C_2$, $C_3$, $C_4$, $C_2\times C_2$, $C_5$, $D_6$,
$(C_2)^3$;
\item[(ii)] $\Delta[G]=\frac{|G|}2$ if and only if $G$ is isomorphic
to one of the groups: $C_6$, $C_8$, $C_2\times C_4$, $D_8$,
$Q_8$.
\end{enumerate}
\end{theorem}

In this theorem by $D_{2n}$ we denote the dihedral group of cardinality
$2n$ and by $Q_8$ the 8-element group of quaternion units.
In \cite{BGN} the difference sizes $\Delta[G]$ was calculated for all
groups $G$ of cardinality $|G|\le 13$.

{
\begin{table}[ht]
\caption{Difference sizes of groups of order $\le13$}\label{tab:BGN}
\begin{tabular}{|c|c|c|c|cc|cc|ccccc|}
\hline
$G$& $C_2$& $C_3$ & $C_5$&$C_4$ &$C_2{\times}C_2$ &$C_6$& $D_6$ & $C_8$
&$C_2{\times}C_4$  & $D_8$ & $Q_8$& $(C_2)^3$\\
\hline
$\Delta[G]$&2&2&3&3&3&3&4&4&4&4&4&5\\
\hline
\hline
$G$&$C_{7}$& $C_{11}$& $C_{13}$ &$C_9$&$C_3{\times}C_3$ &$C_{10}$&
$D_{10}$ & $C_{12}$ &$C_2{\times}C_6$ &$D_{12}$& $A_4$ & $C_3{\rtimes}
C_4$\\
\hline
$\Delta[G]$&3&4&4&4&4&4&4&4&5&5&5&5\\
\hline
\end{tabular}
\end{table}
}

Observing that for each cyclic group $C_n$ of cardinality $n\le13$ the
difference size $\Delta[C_n]$ coincides with the lower bound
$\big\lceil\frac{1+\sqrt{4n-3}}2\big\rceil$ given in
Proposition~\ref{p:BGN}(1), the authors of \cite{BGN} posed the following
problem.

\begin{problem}[Banakh, Gavrylkiv, Nykyforchyn]\label{prob1}
Is $\Delta[C_n]=\big\lceil\frac{1+\sqrt{4n-3}}2\big\rceil$ for any
finite cyclic groups $C_n$?
\end{problem}

Using the results of computer calculations we shall give a negative
answer to  Problem~\ref{prob1}. On the other hand, we shall observe that
the classical difference sets of Singer \cite{Singer} witness that
$\Delta[C_n]=\frac{1+\sqrt{4n-3}}2$ for any number $n=1+q+q^2$ where $q$
is a prime power.

\section{Difference sizes of some special cyclic groups}

In this section we collect some upper bounds on the difference size $\Delta[C_n]$ of a cyclic group  whose order $n$ has some special arithmetic properties, for example, is equal to $q^2+q+1$ or $q^2-1$ for a prime power $q$ or to $p^2-p$ for a prime number $p$.
To derive such upper bounds, we shall use known information on (relative) difference sets.

A subset $D$ of a group $G$ is called a {\em difference set} if each non-idempotent element $g\in G$ can be uniquely written as $g=xy^{-1}$ for some elements $x,y\in D$. This definition implies $|D|^2-|D|=|G|-1$ and hence $|D|=\frac{1+\sqrt{4|G|-3}}2$. The following fundamental result was proved by Singer \cite{Singer} in 1938.

\begin{theorem}[Singer]\label{t:Singer} For any prime power $q$ the
cyclic group $C_{q^2+q+1}$ contains a difference set of cardinality $q+1$ and
hence has difference size $\Delta[C_{q^2+q+1}]=q+1$.
\end{theorem}

Singer's Theorem implies that for any prime power $q$ the cyclic group $C_n$ of cardinality $n=q^2+q+1$ has difference size
$$\Delta[C_n]=q+1=\frac{1+\sqrt{4n-3}}2.$$
So, for infinitely many numbers $n$ the lower bound $\frac{1+\sqrt{4n-3}}2$ for $\Delta[C_n]$ (given in Proposition~\ref{p:BGN}(1)) is attained.

The converse result to Singer's Theorem is known in Algebraic
Combinatorics as PPC (abbreviated from the Prime Power Conjecture).

\begin{conjecture}[$\PPC(n)$] If for a natural number $q<n$ some Abelian group
$G$ of order $|G|=q^2+q+1$ contains a difference set $D\subset G$, then
$q$ is a prime power.
\end{conjecture}

In \cite{G} $\PPC(n)$ is confirmed for all numbers $n<2\,000\,000$.
The Prime Power Conjecture implies the following converse to
Theorem~\ref{t:Singer}.

\begin{proposition} Let $G$ be an Abelian group of difference size
$\Delta[G]=\frac{1+\sqrt{4|G|-3}}2$. If $\PPC(|G|)$ holds, then
$q=\Delta[G]-1$ is a prime power and $|G|=q^2+q+1$.
\end{proposition}

\begin{proof} Let $q:=\Delta[G]-1$. The equality
$1+q=\frac{1+\sqrt{4|G|-3}}2$ implies that
$|G|=\frac14\big((2q+1)^2+3\big)=q^2+q+1$. Let $D\subset G$ be a
difference basis of cardinality $|D|=\Delta[G]=1+q$ in $G$. Consider the
surjective map $\xi:D\times D\to G$, $\xi:(x,y)\mapsto xy^{-1}$. The
equality $|D|^2-|D|=q^2+q=|G\setminus\{1_G\}|$ implies that for each $g\in
G\setminus\{1_G\}$ the preimage $\xi^{-1}(g)$ is a singleton, which means
that $D$ is a difference set in $G$. Now $\PPC(|G|)$ implies that $q$ is a prime power.
\end{proof}

Singer derived his theorem studying properties of projective planes over finite fields. A corresponding result for affine planes was obtained by Bose \cite{Bose} and Chowla \cite{Chowla}.

\begin{theorem}[Bose-Chowla]\label{t:Bose} For any prime power $q$ the set $C_{q^2-1}\setminus C_{q-1}$  has a difference basis of cardinality $q$ in the cyclic group $C_{q^2-1}$. Consequently, $\Delta[C_{q^2-1}\setminus C_{q-1}]=q$ and
$$\Delta[C_{q^2-1}]\le\Delta[C_{q^2-1}\setminus C_{q-1}]+\Delta[C_{q-1}]-1=q-1+\Delta[C_{q-1}].$$
\end{theorem}

We shall also need the following theorem essentially proved by Rusza \cite{Rusza}.

\begin{theorem}[Rusza]\label{t:Rusza} For any prime number $p$ the set $C_{p^2-p}\setminus (C_p\cup C_{p-1})$ has a difference basis of cardinality $p-1$ in the group $C_{p^2-p}$.  Consequently, $\Delta[C_{p^2-p}\setminus (C_{p}\cup C_{p-1})]=p-1$ and
$$\Delta[C_{p^2-p}]\le p-3+\Delta[C_p]+\Delta[C_{p-1}].$$
\end{theorem}

\begin{proof} Given a prime number $p$, consider the field $\IZ_p=\IZ/p\IZ$ of residues modulo $p$. It is known that the multiplicative group $\IZ_p^*=\IZ_p\setminus\{0\}$ of this field is isomorphic to the cyclic group $C_{p-1}$. Since the numbers $p$ and $p-1$ are relatively prime, the product $\IZ_p\times \IZ_p^*$ is isomorphic to the cyclic group $C_{p^2-p}$.
So, instead of the group $C_{p^2-p}$, we can consider the group $\IZ_p\times\IZ_p^*$ (which is the direct product of the additive and multiplicative groups of the field $\IZ_p$).

We claim that the set $B=\{(x,x):x\in\IZ_p^*\}$ is a difference basis for the set $A=\{(x,y)\in\IZ_p\times\IZ_p^*:x\ne 0\mbox{ and }y\ne 1\}$ in the group $\IZ_p\times\IZ_p^*$. Given any pair $(a,b)\in A$ we need to find two elements $x,y\in\IZ_p^*$ such that $(x-y,xy^{-1})=(a,b)$. Since $b\ne 1$, the element $b-1$ is invertible in the field $\IZ_p^*$, so we can consider the elements $x=b(b-1)^{-1}a\ne 0$ and $y=(b-1)^{-1}a\ne 0$. The pairs $(x,x)$ and $(y,y)$ belong to the set $B$ and their difference $(x-y,xy^{-1})$ in the group $\IZ_p\times\IZ_p^*$ is equal to $(a,b)$, witnessing that $B$ is a difference basis for the set $A$.

Observe that the complement of the set $A$ in the group $\IZ_p\times\IZ_p^*$ is equal to the union of two subgroups $\{0\}\times\IZ_p^*$ and $\IZ_p\times\{1\}$.
Therefore, $$\Delta[C_{p^2-p}\setminus(C_p\cup C_{p-1})]=\Delta[A]\le|B|=p-1.$$
The upper bound $\Delta[A]\le p-1$ combined with the equality $$
\begin{aligned}
|A|&=|\IZ_p\times\IZ_p^*|-|\{0\}\times\IZ_p^*|-|\IZ_p\times\{1\}|+|\{(0,1)\}|=\\
&=(p^2-p)-(p-1)-p+1=p^2-3p+2=(p-1)(p-2)=|B|(|B|-1)
\end{aligned}
$$ implies that $\Delta[C_{p^2-p}\setminus(C_p\cup C_{p-1})]=\Delta[A]=p-1$.

By Proposition~\ref{p:BGN}(4),
$$
\begin{aligned}
\Delta[C_{p^2-p}]&=\Delta[\IZ_p\times\IZ_p^*]=\Delta[A\cup(\IZ_p\times\{1\})\cup(\{0\}\times\IZ_p^*)]\le
\Delta[A]+\Delta[\IZ_p\times\{1\}]+\Delta[\{0\}\times \IZ_p^*]-2\le\\
&\le |B|+\Delta[\IZ_p]+\Delta[\IZ_p^*]-2=p-3+\Delta[C_p]+\Delta[C_{p-1}].
\end{aligned}
$$
\end{proof}

\section{Difference sizes of number intervals}

In this section we apply known information on difference sizes of number intervals to evaluating the difference sizes of finite cyclic groups. For integer numbers $a<b$ by $[a,b]$ we shall denote the order-interval $\{x\in\IZ:a\le x\le b\}$ in the group $\IZ$ of integer numbers.

For a natural number $n\in\IN$ by $\Delta[n]$ we shall denote the difference size of the interval $[0,n]$. It is equal to the difference size of the intervals $[1,n]$ and $[-n,n]$. Also we put $\eth[n]=\frac{\Delta[n]}{\sqrt{n}}$.

For example, the interval $[0,6]$ has difference size $\Delta[6]=4$ as witnessed by difference basis $\{0,1,4,6\}$. It is clear that $\Delta[n]$ is a non-decreasing function of the integer parameter $n$ and $\Delta[n](\Delta[n]-1)\ge n$, which implies that $\Delta[n]>\sqrt{n}$ for all $n\in\IN$.

Difference bases for the order intervals $[0,n]$ were studied by R\'edei and R\'enyi \cite{RR} who  proved that the limit $\lim_{n\to\infty}\eth[n]$ exists and is equal to $\inf_{n\in\IN}\eth[n]$. Moreover,
$$2.424...=2+\frac{4}{3\pi}\le \lim_{n\to\infty}\eth[n]^2=\inf_{n\in\IN}\eth[n]^2\le \eth[6]^2=\frac83=2.666...$$

These lower and upper bounds were improved by Leech \cite{Leech} and Golay \cite{Golay} who proved the following theorem.

\begin{theorem}[Leech-Golay]\label{t:LG} For any natural number $n$ we get the lower and upper bounds:
$$2.434...=2+\max_{0<\varphi<\pi}
\frac{2\sin(\varphi)}{\varphi+\pi}\le \lim_{n\to\infty}\eth[n]^2=\inf_{n\in\IN}\eth[n]^2\le \eth[6166]^2\le\frac{128^2}{6166}=2.6571...$$
\end{theorem}
For small numbers $n$ the difference sizes $\Delta[n]$ of the intervals $[0,n]$ have been calculated by computer. The following table (taken from \cite{Leech} and \cite{Wiki}) gives the values of $\Delta[n]$ for numbers $n\le 61$ such that $\Delta[n+1]>\Delta[n]$.

\begin{table}[ht]
\caption{The values of $\Delta[n]$ for small $n$.}\label{tab:ruler1}
\begin{tabular}{|r|c|c|c|c|c|c|c|c|c|c|c|c|}
\hline
$n=$&1& 3&6&9&13&18&24&29&37&45&51&61\\
$\Delta[n]=$&2&3&4&5&6&7&8&9&10&11&12&13\\
$\eth[n]^2\approx$ &4&3&{2.666}&{2.777}&{2.769}&{2.722}&{2.666}& 2.793&{2.703}&{2.688}&{2.824}&{2.770}\\
\hline
\end{tabular}
\end{table}

This table shows that for $n\le 61$ the smallest value $\frac83$ of the difference characteristic $\eth[n]^2$ is attained for $n\in\{6,24\}$. Combining the difference basis $\{0,1,4,6\}$ for $[0,6]$ with the Singer difference sets, Leech \cite{Leech} and Golay \cite{Golay} have found four larger numbers $n$ with $\eth[n]^2<\frac83$. These (four) numbers are presented in Table~\ref{tab:ruler2}.

\begin{table}[ht]
\caption{Some numbers $n$ with $\eth[n]^2<{\frac83}$.}\label{tab:ruler2}
\begin{tabular}{|r|c|c|c|c|}
\hline
$n=$&4064&4713&5416&6166\\
$\Delta[n]\le$&104&112&120&128\\
$\eth^2[n]\le$&2.6615&2.6616&2.6588&2.6572\\
\hline
\end{tabular}
\end{table}

Difference sizes of the intervals yield upper bounds on the difference sizes of cyclic groups.

\begin{proposition}\label{p:c<i} For any natural number $n$ we get
$$\Delta[C_n]\le \Delta[k]\mbox{ \ and \ } \eth[C_n]\le\frac{\eth[k]}{\sqrt{2}}$$ where $k=\lceil\frac{n-1}2\rceil.$
\end{proposition}

\begin{proof} Given a natural number $n$, consider the homomorphism $\gamma:\IZ\to C_n$, $\gamma:t\mapsto e^{\frac{2\pi i}nt}$. For the number $k=\lceil\frac{n-1}2\rceil$, choose a subset $D\subset\IZ$ of cardinality $|D|=\Delta[k]$ such that $D-D$ contains the interval  $[-k,k]$.  Taking into account that $n\le 2k+1=|[-k,k]|$, we conclude that $\gamma(D)\gamma(D)^{-1}=\gamma(D-D)\supset\gamma([-k,k])=C_n$, which means that the set $B=\gamma(D)$ is a difference basis for the group $C_n$ and hence $\Delta[C_n]\le|B|\le|D|=\Delta[k]$.

Observe that $k=\big\lceil\frac{n-1}2\big\rceil\le\frac{n}2$ and hence
$$\eth[C_n]=\frac{\Delta[C_n]}{\sqrt{n}}\le \frac{\Delta[k]}{\sqrt{k}}\frac{\sqrt{k}}{\sqrt{n}}= \eth[k]\cdot\sqrt{\frac{k}n}\le\eth[k]\sqrt{\frac{n/2}{n}}=\frac{\eth[k]}
{\sqrt{2}}.$$
\end{proof}

Proposition~\ref{p:c<i} and Theorem~\ref{t:LG} allow us to evaluate lower and
upper limits of the difference characteristics of cyclic groups.

\begin{corollary}\label{c:l+ub} For every natural number $n$ we have the lower and upper bounds:
$$1=\inf_{n\in\IN}\eth[C_n]=
\liminf_{n\to\infty}\eth[C_n]\le \limsup_{n\to\infty}\eth[C_n]\le \frac1{\sqrt{2}}\cdot\inf_{n\in\IN}\eth[n]\le \frac{\eth[6166]}{\sqrt{2}}\le \frac{64}{\sqrt{3083}}=1.1526...<\frac2{\sqrt{3}}=1.1547...$$
\end{corollary}

Corollary~\ref{c:l+ub} implies that $\eth[C_n]<\frac2{\sqrt{3}}$ for all sufficiently large $n$. In Theorem~\ref{t:huge} we shall show that this upper bound holds for all $n\ge 2\cdot 10^{15}$.

At first we find some upper bounds of the difference sizes of the intervals $[-n,n]$, using the approach first exploited by R\'edei and R\'enyi \cite{RR}, and then developed by Leech \cite{Leech} and Golay \cite{Golay}.

To write down these upper bounds, we shall need some information on the numbers $\delta_k[C_m]$, which are defined as follows. For a natural number $m$ and a non-negative number $k<\Delta[C_m]$ let $\delta_k[C_m]$ be the largest integer number $d<m$ for which there exists a set $B\subset [0,m]$ of cardinality $|B|=\Delta[C_m]$ such that $B-B+m\IZ=\IZ$ and $|B\cap[0,d)|\le k$. In \cite{Golay} the numbers $\delta_0[C_m]$ were denoted by $c_{\max}$ and were calculated for all cyclic groups $C_m$ of order $m=1+q+q^2$ where $q\le32$ is a prime power:
\begin{table}[ht]
\caption{The numbers $\delta_0[C_m]$ for cyclic groups of order $m=1+q+q^2$ for a prime power $q$.}\label{tab:delta0}
\begin{tabular}{|c|c|c|c|c|c|c|c|c|c|c|c|c|c|c|c|c|c|c|}
\hline
$q$&2&3&4&5&7&8&9&11&13&16&17&19&23&25&27&29&31&32\\
$m$&7&13&21&31&57&73&91&133&183&273&307&381&553&651&757&871&993&1057\\
$\delta_0[C_m]$&4&7&10&14&22&28&36&48&56&72&91&98&128&159&172&191&209&198\\
\hline
\end{tabular}
\end{table}

This table is completed by Table~\ref{tab:delta-k} giving the values of
$\delta_k[C_m]$ for positive $k<m\le 307$. These values are found by computer.

\begin{table}[ht]
\caption{The values of the numbers $\delta_k[C_m]$ for some $k$
and $m$.}\label{tab:delta-k}
\begin{tabular}{|cc|c|c|c|c|c|c|c|c|c|c|c|c|c|c|c|c|c|}
\hline
&$k$&0&1&2&3&4&5&6&7&8&9&10&11&12&13&14&15&16\\
$m$&&&&&&&&&&&&&&&&&&\\
\hline
7&&4&6&7&-&-&-&-&-&-&-&-&-&-&-&-&-&-\\
13&&7&10&12&13&-&-&-&-&-&-&-&-&-&-&-&-&-\\
21&&10&13&17&20&21&-&-&-&-&-&-&-&-&-&-&-&-\\
31&&14&19&25&28&30&31&-&-&-&-&-&-&-&-&-&-&-\\
39&&18&23&29&35&37&38&39&-&-&-&-&-&-&-&-&-&-\\
57&&22&29&38&43&50&54&56&57&-&-&-&-&-&-&-&-&-\\
73&&28&38&41&52&58&66&70&72&73&-&-&-&-&-&-&-&-\\
91&&36&46&51&62&71&74&84&88&90&91&-&-&-&-&-&-&-\\
133&&48&61&68&79&94&102&114&120&125&130&132&133&-&-&-&-&-\\
183&&56&69&92&98&109&128&139&153&158&168&175&180&182&183&-&-&-\\
273&&72&92&110&130&145&152&175&184&202&218&226&242&258&266&270&272&273\\
307&&91&106&128&152&160&171&199&214&232&245&261&276&284&292&300&304&306\\
\hline
\end{tabular}
\end{table}

\begin{theorem}\label{t:lb-delta0} For any natural number $m\ge 3$ we get the lower bound $$\delta_0[C_m]\ge \Big\lfloor\frac{m}{\Delta[C_m]}\Big\rfloor.$$ If $m=1+q+q^2$ for some prime power $q$, then $$\delta_0[C_m]\ge\max\Big\{q+2,\Big\lfloor\frac{q^2+q+1}{q-\sqrt{q}+3}\Big\rfloor\Big\}\ge
\max\{q+2,q+\sqrt{q}-3\}.$$
\end{theorem}

\begin{proof} Given a natural number $m\ge 3$, fix a difference basis $D\subset C_m$ of cardinality $|D|=\Delta[C_m]$. We lose no generality assuming that $D$ contains the unit $1$ of the group $C_m$. It follows that $|D|<m$. Let $g=e^{\frac{2\pi}mi}$ be a generator of the cyclic group $C_m$. Let $d=\big\lfloor\frac{m}{\Delta[C_m]}\big\rfloor$ and consider the interval $I=\{g^j:0\le j<d\}\subset C_m$, containing $d$ elements of $C_m$.

\begin{claim}\label{cl:empty} There exists $c\in C_m$ such that $cI\cap D=\emptyset$.
\end{claim}

\begin{proof}
For a subset $A\subset C_m$ denote by $\chi_A:C_m\to\{0,1\}$ its characteristic function (which means that $\chi_A^{-1}(1)=A$).
Assume that $cI\cap D\ne \emptyset$ for all $c\in C_m$ and observe that
$$
\begin{aligned}
m&\le\sum_{c\in C_m}|cI\cap D|=\sum_{c\in C_m}\sum_{x\in C_m}\chi_{cI}(x)\chi_D(x)=\sum_{x\in C_m}\chi_D(x)\sum_{c\in C_m}\chi_{I}(c^{-1}x)=\\
&=\sum_{x\in C_m}\chi_D(x)|I|=|I|\cdot|D|=d\cdot \Delta[C_m]\le m,
\end{aligned}
$$which implies that $d\cdot\Delta[C_m]=m$ and $|cI\cap D|=1$ for all $c\in C_m$. The latter property of $D$ can be used to show that $D$ coincides with the subgroup $H$ generated by $g^d$, which is not possible as $DD^{-1}=G\ne H$.
\end{proof}

By Claim~\ref{cl:empty}, there exists $c\in C_m$ such that $cI\cap D$ is empty.
Then $B=c^{-1}D$ is a difference basis for $C_m$ with $I\cap B=\emptyset$. It
follows that the set $A=\{a\in \IZ:0<a\le m,\;\;g^a\in B\}$ has
$A\cap[0,d)=\emptyset$ and witnesses that $\delta_0[C_m]\ge
d=\big\lfloor\frac{m}{\Delta[C_m]}\big\rfloor$.
\smallskip

Now assume that $m=1+q+q^2$ for some prime power $q$. In this case $D$ is a difference set of cardinality $|D|=1+q$. We recall that $C_m$ is the multiplicative subgroup $\{z\in\IC:z^m=1\}$ of the unit circle $\IT=\{z\in\IC:|z|=1\}$ on the complex plane. 
Let $I:=\{e^{i\phi}:0\le\phi<\pi\}\subset\IT$ be the upper  half-circle and observe that $\IT=I\cup(-I)$.

\begin{claim}\label{cl} For some $z\in \IT$ the set $I\cap zD$ has cardinality $|I\cap zD|\le\frac{1+q-\sqrt{q}}2$.
\end{claim}

\begin{proof} For a subset $A\subset \IT$ by $\chi_A:\IT\to\{0,1\}$ we denote the characteristic function of the set $A$ in $\IT$, which means that $\chi_A^{-1}(1)=A$. Observe that each element $z\in C_m\cap(iI^{-1})$ has positive real part and the sum
$r:=\sum_{z\in C_m\cap iI^{-1}}z$ is a positive real number.

Now consider the complex number $\delta:=\sum_{z\in C_m}\chi_D(z)z$ and observe that
$$|\delta|^2=\delta\bar\delta=\sum_{x,y\in C_m}\chi_D(x)\chi_D(y)x\bar y=\sum_{x\in C_m}\chi_D(x)^2xx^{-1}+\sum_{x\ne y}\chi_D(x)\chi_D(y)xy^{-1}=|D|+\sum_{z\in C_m\setminus\{1\}}z=|D|-1=q.$$
Then $\sum_{z\in D}z=\delta=\sqrt{q}\,e^{-i\psi}$ for some $\psi\in\IR$ and hence the set $D_\psi:=De^{i\psi}$ has $\sum_{z\in D_\psi}z=\sqrt{q}$.

Let $C_\psi=C_me^{i\psi}\subset\IT$. Observe that
$$
\begin{aligned}
r\cdot\sqrt{q}&=\sum_{x\in C_m}\chi_{iI^{-1}}(x)x\cdot\sum_{y\in C_\psi}\chi_{D_\psi}(y)y=\sum_{z\in C_{\psi}}\sum_{{(x,y)\in C_m\times C_\psi}\atop{xy=z}}\chi_{iI^{-1}}(x)\cdot\chi_{D_\psi}(y)\cdot z=\\
&=\sum_{z\in C_\psi}\sum_{y\in C_\psi}\chi_{iI^{-1}}(zy^{-1})\cdot\chi_{D_\psi}(y)\cdot z=\sum_{z\in C_\psi}\sum_{y\in C_\psi}\chi_{zi^{-1}I}(y)\cdot\chi_{D_\psi}(y)\cdot z=\sum_{z\in C_\psi}|i^{-1}zI\cap De^{i\psi}|\cdot z.
\end{aligned}
$$

Assuming that $|I\cap zD|>c:=\frac{1+q-\sqrt{q}}2$ for all $z\in \IT$, we conclude that $$|i^{-1}zI\cap De^{i\psi}|=|D|-|izI\cap De^{i\psi}|<(1+q)-c$$ for all $z\in\IT$. Taking into account that each complex number $z\in iI$ has negative real part $\Re(z)$,
$$0<\sum_{z\in C_\psi\cap(-iI)}\Re(z)\le\sum_{z\in C_m\cap(-iI)}\Re(z)=r$$ and $$0=\sum_{z\in C_\psi}\Re(z)=\sum_{z\in C_\psi\setminus iI}\Re(z)+\sum_{z\in C_\psi\cap iI}\Re(z),$$we conclude that
$$
\begin{aligned}
r\sqrt{q}&=\sum_{z\in C_\psi}|i^{-1}zI\cap De^{i\psi}|\cdot z=\sum_{z\in  C_\psi\cap(-iI)}|i^{-1}zI\cap De^{i\psi}|\cdot\Re(z)+\sum_{z\in C_\psi\cap(iI)}|i^{-1}zI\cap De^{i\psi}|\cdot\Re(z)<\\
&<(1+q-c)\sum_{z\in C_\psi\cap (-iI)}\Re(z)+c\sum_{z\in C_\psi\cap iI}\Re(z)=(1+q-2c)\sum_{z\in C_\psi\setminus iI}\Re(z)\le(1+q-2c)r,
\end{aligned}
$$
which implies that $\sqrt{q}<(1+q-2c)$ and hence $c<\frac{1+q-\sqrt{q}}2$. But this contradicts the definition of the constant $c$.
\end{proof}

By Claim~\ref{cl}, there exists a complex number $c\in \IT$ such that $|I\cap cD|\le\frac{1+q-\sqrt{q}}2$. Let $l=|I\cap cD|+1$ and $J=\{e^{ it}:0\le t<\pi/l\}\subset\IT$. Observe that the arc $I$ can be covered by $l>|I\cap cD|$ disjoint copies of the arc $J$. By the Pigeonhole Principle, for some $z\in\IT$ the arc $zJ$ is disjoint with the set $cD$. Then the arc $c^{-1}zJ$ is disjoint with the set $D$.
It is easy to see that the arc $c^{-1}zJ$ contains at least $k:=\lfloor\frac{m}{2l}\rfloor$ consecutive points of the group $C_m$. Therefore, the set $C_m\setminus D$ contains $k$ consecutive points. Replacing $D$ by a suitable shift, we can assume that those $k$ consecutive points form the set $\{g^j:0\le j<k\}$ where $g=e^{\frac{2\pi}mi}$ is the generator of the cyclic group $C_m$. Then the set $A=\{j\in[0,m]:g^j\in D\}$ is disjoint with the set $[0,k)$ and witnesses that
$$
\begin{aligned}
\delta_0[C_m]&\ge k=\left\lfloor\frac{m}{2l}\right\rfloor\ge \frac{m}{2l}=\Big\lfloor\frac{m}{2(|I\cap cD|+1)}\Big\rfloor=\Big\lfloor\frac{m}{1+q-\sqrt{q}+2}\Big\rfloor=\Big\lfloor\frac{q^2+q+1}{q-\sqrt{q}+3}\Big\rfloor>\\
&\frac{q^2+q+1}{q-\sqrt{q}+3}-1=
q+\sqrt{q}-2-\frac{4(\sqrt{q}-1)}{q-\sqrt{q}+3}>q+\sqrt{q}-3.
\end{aligned}
$$
If $q\ge 16$, then $$\delta_0(C_m)\ge \Big\lfloor\frac{q^2+q+1}{q-\sqrt{q}+3}\Big\rfloor\ge
\Big\lfloor\frac{q^2+q+1}{q-1}\Big\rfloor=q+2.$$
 If $q<16$, then the equality $\delta_0[C_m]\ge q+2$  follows from Table~\ref{tab:delta-k}.
\end{proof}

The numbers $\delta_k[C_m]$ are used in the following theorem giving an upper bound for the difference sizes of intervals.

\begin{theorem}\label{t:ub-Delta-n} For any non-negative integer numbers $n,m,k$ with $k<m$ we get
the upper bound
$$\Delta\big[nm+\delta_k[C_m]-1\big]\le \Delta[n]\cdot\Delta[C_m]+k.$$
\end{theorem}

\begin{proof}  Fix a difference basis $D\subset\IZ$ for the interval $[-n,n]$ of cardinality $|D|=\Delta[n]$.

By the definition of the number $\delta_k[C_m]$, there exists a set  $A\subset
[0,m]$ of cardinality $|A|=\Delta[C_m]$ such that $A-A+m\IZ=\IZ$ and $|A\cap
[0,\delta_k[C_m])|\le k$. Find two numbers $\lambda,l\in D$ with $\lambda-l=n$.
It is clear that the set $$B:=\{a+md:a\in A,\;d\in D\}\cup\{a+m(\lambda+1):a\in
A\cap[0,\delta_k[C_m])\}$$ has cardinality $$|B|\le |D|\cdot
|A|+|A\cap[0,\delta_k[C_m])|\le\Delta[n]\cdot\Delta[C_m]+k.$$ We claim that the
interval $J=\{x\in \IZ:|x|<mn+\delta_k[C_m]\}$ is contained in $B-B$. Since the
set $B-B$ is symmetric, it suffices to show that each positive number $x\in J$
is contained in $B-B$. Write $x$ as $x=my+z$ for some integer numbers $y,z$
such that $0\le y\le n$ and $0\le z<m$. By the choice of $A$, there are numbers
$a,b\in A$ such that $z=a-b+mj$ for some $j\in\IZ$. Taking into account that
$|mj|=|a-b-z|\le |a-b|+|z|<m+m=2m$, we conclude that $|j|\le 1$ and hence
$|y+j|\le n+1$.

It follows that $x=my+z=m(y+j)+a-b$.
If $|y+j|\le n$, then we can choose two numbers $u,v\in D$ such that $y+j=u-v$ and conclude that $$x=m(y+j)+z=m(u-v)+a-b=(mu+a)-(mv+b)\in B-B.$$
So, we assume that $|y+j|=n+1$. Then $x=m(y+j)+a-b=m(n+1)+a-b$. Taking into account that $x<mn+\delta_k[C_m]$, we conclude that $a\le m+a-b=x-mn<\delta_k[C_m]$, and hence $a+m(\lambda+1)\in B$. Then $x=m(n+1)+a-b=m(\lambda-l+1)+a-b=a+m(\lambda+1)-(b+ml)\in B-B$.
Therefore $J\subset B-B$ and $\Delta[nm+\delta_k[C_m]-1]\le |B|\le \Delta[n]\cdot\Delta[C_m]+k$.
\end{proof}

\begin{corollary}\label{c:Delta-Cl} Let $n$ be a natural number, $q$ is a prime power and $k=n(1+q+q^2)+q+1$. For any natural number $l\le 2k+1$ we get
the upper bound
$$\Delta[C_l]\le \Delta[k]\le \Delta[n]\cdot(q+1).$$
\end{corollary}

\begin{proof}  By Theorem~\ref{t:Singer}, the cyclic group $C_m$ of order $m=1+q+q^2$ has difference size $\Delta[C_m]=q+1$. By Theorem~\ref{t:lb-delta0}, $\delta_0[C_m]\ge q+2$.
By Theorem~\ref{t:ub-Delta-n},
$$\Delta[C_l]\le \Delta[k]\le\Delta[nm+\delta_0[C_m]-1]\le \Delta[n]\cdot\Delta[C_m]=\Delta[n]\cdot(q+1).$$
\end{proof}

Applying Corollary~\ref{c:Delta-Cl} with $n=6$ we derive another  corollary.

\begin{corollary}\label{c:ub-Cl1} For any prime power $q$ and a natural number $l$ with $l\le 15+14q+12q^2$ we get
the upper bound $\Delta[C_l]\le 4(q+1)$.
\end{corollary}

For a real number $x$ by $q(x)$ we denote the smallest prime power, which is
larger or equal than $x$. It is easy to see that for real numbers $x\ge 12$ and
$y\ge0$ the inequality $x\le 15+14y+12y^2$ is equivalent to
$y\ge\frac{-7+\sqrt{12x-131}}{12}$. This observation, combined with
Corollary~\ref{c:ub-Cl1} yields the following upper bound for $\Delta[C_n]$.

\begin{corollary}\label{c:ub-Cl2} Each finite cyclic group $C_n$ of order $n\ge 11$ has difference size $\Delta[C_n]\le 4+4q\big(\frac{-7+\sqrt{12n-131}}{12}\big)$.
\end{corollary}

For a real number $x$ by $p(x)$ we denote the smallest prime number greater or equal to $x$. It is clear that $x\le q(x)\le p(x)$.
By \cite{BHP}, $p(x)=x+O(x^{21/40})$.
The (still unproven) Andrica's Conjecture \cite{And} says that $p(x)<x+2\sqrt{x}+1$ for all $x\ge 1$. This conjecture was confirmed by Imran Ghory \cite{Ghory} for all $x\le 1.3\times 10^{16}$.

\begin{corollary}\label{c:andrica} Each finite cyclic group $C_n$ of cardinality $11\le n\le 2\cdot 10^{33}$ has difference size $$\Delta[C_n]<\frac1{3}\sqrt{12n-131}+
\frac4{\sqrt{3}}\sqrt[4]{12n-131}+\frac{17}3.$$
If the Andrica Conjecture is true, then this upper bound holds for all numbers $n$.
\end{corollary}

\begin{proof} Given a number $n\le 2\cdot 10^{33}$, consider the real number $x:=\frac{-7+\sqrt{12n-131}}{12}$. The inequality $n\le 2\cdot 10^{33}$ implies that $x\le 1.3\cdot 10^{16}$, so we can apply the result of Ghory \cite{Ghory}, and conclude that  $p(x)<x+2\sqrt{x}+1$. By Corollary~\ref{c:ub-Cl2},
$$
\begin{aligned}
\Delta[C_n]&\le 4(p(x)+1)<4x+8\sqrt{x}+8= \frac{-7+\sqrt{12n-131}}3+8\sqrt{\frac{-7+\sqrt{12n-131}}{12}}+8<\\
&<
\frac13\sqrt{12n-131}+\frac{4}{\sqrt{3}}\sqrt[4]{12n-131}+\frac{17}3.
\end{aligned}
$$
If the Andrica's Conjecture is true, then the same argument works for all $n$.
\end{proof}

Applying Corollary~\ref{c:ub-Cl2} with $n=6166$ and the upper bound $q(x)\le p(x)=x+O(x^{21/40})$ from \cite{BHP}, we get a more refined upper bound for $\Delta[C_m]$.

\begin{theorem}\label{t:huge} For any prime power $q$ and a natural number $n\le 12335+12334q+12332q^2$ we get
the upper bound $\Delta[C_n]\le 128(q+1)$. Consequently, for any $n\ge 926$ we get the upper bound
$$\Delta[C_n]\le 128+128\cdot q\big(\tfrac{-6167+\sqrt{12332n-114083331}}{12332}\big)=\frac{64}{\sqrt{3084}}\sqrt{n}+O(n^{21/80}).$$

If $n\ge 2\cdot 10^{15}$, then $\Delta[C_n]<\frac2{\sqrt{3}}\sqrt{n}$.
\end{theorem}

\begin{proof} By Theorem~\ref{t:Singer}, the cyclic group of order $m=1+q+q^2$ has difference size $\Delta[C_m]=q+1$. By Theorem~\ref{t:lb-delta0}, $\delta_0[C_m]\ge q+2$. Since $n\le 12335+12334q+12332q^2=2(6166(1+q+q^2)+q+1)+1$, we can apply Theorem~\ref{t:ub-Delta-n} and obtain the upper bound
$$\Delta[C_n]\le \Delta[6166 m+\delta_0[C_m]+1]\le \Delta[6166]\cdot\Delta[C_m]=128(q+1).$$

If $n\ge 926$, then the real number $x:=\frac{-6167+\sqrt{12332n-114083331}}{12332}$ is well-defined and $$n=12335+12334\,x+12332\,x^2\le 12335+12334\,q(x)+12332\,q(x)^2.$$ By \cite{BHP}, $q(x)\le p(x)=x+O(x^{21/40})$. Then
$$\Delta[C_n]\le128(q(x)+1)=128 x+O(x^{21/40})=\frac{128}{\sqrt{12332}}\sqrt{n}+O(n^{21/80})= \frac{64}{\sqrt{3083}}\sqrt{n}+O(n^{21/80}).$$

Now assume that $n\ge 2\cdot 10^{15}$. In this case $x=\frac{-6167+\sqrt{12332n-114083331}}{12332}>396\,738$. By \cite[6.8]{Dusart}, $p(x)\le x\big(1+\frac1{25\ln^2 x}\big)$. Then
$\Delta[C_n]\le 128(p(x)+1)$ and we see that the inequality $\Delta[C_n]<\frac2{\sqrt{3}}\sqrt{n}$, follows from the inequality
$$128+32\cdot
\frac{-6167+\sqrt{12332n-114083331}}{3083}
\big(1+\frac1{25\ln^2(\tfrac{-6167+\sqrt{12332n-114083331}}{12332})}\big)
<\frac2{\sqrt{3}}\sqrt{n},$$
holding for all $n\ge 2\cdot 10^{15}$.
\end{proof}

Now we evaluate the difference sizes of intervals $[0,n]$ and cyclic groups $C_n$ for relatively small $n$.

Applying Theorem~\ref{t:ub-Delta-n} to the known values of the difference sizes $\Delta[3]=3$, $\Delta[4]=6$, $\Delta[8]=24$, known values of difference sizes of cyclic groups given in Table~\ref{tab:cycl}, and known values of the numbers $\delta_k[C_m]$ given in Tables~\ref{tab:delta0} and \ref{tab:delta-k}, we obtain the upper bounds for the difference sizes $\Delta[n]$ of intervals of length $n\le 6166$, given in Table~\ref{tab:ruler3}.

\begin{table}[ht]
\caption{The values of $\Delta[n]$ for relatively small $n$.}\label{tab:ruler3}
\begin{tabular}{|r|c|c|c|c|c|c|c|c|c|c|c|c|c|c|c|c|c|}
\hline
$n=$&72&84&87&\!105\!&109&135&138&142&145&200&204&208&211&251&256&262&268\\
$\Delta[n]\le$&15&16&17&18&19&20&21&22&23&24&25&26&27&28&29&30&31\\
\hline
\end{tabular}
\begin{tabular}{|r|c|c|c|c|c|c|c|c|c|c|c|c|c|c|c|}
\hline
$n=$&363&465&581&684&845&\!1153\!&1389&1709&1932&2383&3445&\!4064\!&\!4713\!&\!5416\!&\!6166\!\\
$\Delta[n]\le$&32&36&40&45&48&56&64&68&72&80&96&104&112&120&128\\
\hline
\end{tabular}
\end{table}

Now we can prove an upper bound for $\eth[C_n]$, holding for small $n$.

\begin{theorem}\label{t:ub-Cn} The upper bound $\eth[C_n]\le\frac{12}{\sqrt{73}}=1.40449...$ holds for all $n\ge 9$.
Moreover, if $n\ne 292$, then $\eth[C_n]\le\frac{24}{\sqrt{293}}=1.40209...$.
\end{theorem}

\begin{proof} For $n\in [9,100]$ the inequality $\eth[C_n]\le\frac{24}{\sqrt{293}}$ can be verified using known values of $\Delta[C_n]$, see Table~\ref{tab:cycl}.

For $n\in[101,11\,981]$ the upper bound $\eth[C_n]\le\frac{12}{\sqrt{73}}$ follows from the upper bound $\Delta[C_n]\le \Delta[\big\lceil\frac{n-1}2\rceil\big]$ given in Proposition~\ref{p:c<i} and the upper bounds for the difference sizes $\Delta[n]$ given in Table~\ref{tab:ruler3}. For example, let us prove the upper bound $\eth[C_n]\le\frac{12}{\sqrt{73}}$ for the number $n=292$.

By Proposition~\ref{p:c<i}, $\Delta[C_{292}]\le \Delta[146]$.
Looking at Table~\ref{tab:ruler1}, we see that $\Delta[146]\le\Delta[200]\le 24$. Consequently, $\eth[C_{292}]=\frac{\Delta[C_{292}]}{\sqrt{292}}\le \frac{24}{\sqrt{4\cdot 73}}=\frac{12}{\sqrt{73}}$. For all other numbers $n\in[101,11\,989]$ by the same method we get the inequality $\eth[C_n]\le\frac{24}{\sqrt{293}}$.

For $n\ge 11\,982$ we shall apply the known fact (see the sequence
https://oeis.org/A166968 on the
 On-line Encyclopedia of Integer Sequences) saying that for every $k\ge 32$ the interval $[k,\frac{7}{6}k]$ contains a prime number.
Consider the real number $x=\frac{-7+\sqrt{12n-131}}{12}>31$ and let $q$ be the smallest prime power which is greater or equal to $x$. It follows that $$q\le \tfrac76\,\lceil x\rceil<\tfrac76\,(x+1)=\frac76\cdot\frac{5+\sqrt{12n-131}}{12}.$$
The inequality $q\ge x$ implies $n\le 15+14q+12q^2$. By Corollary~\ref{c:ub-Cl2}, $\Delta[C_n]\le 4(q+1)\le \frac76\cdot\frac{5+\sqrt{12n-131}}3+4$. To prove that $\eth[C_n]\le \frac{24}{\sqrt{293}}$, it remains to check that $\frac76\cdot\frac{5+\sqrt{12n-131}}3+4\le \frac{24}{\sqrt{293}}\sqrt{n}$.
The elementary calculations show that this inequality holds for all $n\ge 11\, 436$.
\end{proof}

Theorem~\ref{t:ub-Cn} and known values of $\eth[C_n]$ for $n<9$ allow us to find the largest value of the difference characteristics $\eth[C_n]$.

\begin{corollary}\label{t:max}$\max_{n\in\IN}\eth[C_n]=\eth[C_4]=\frac32$.
\end{corollary}

Also we can establish some upper bounds for the difference sizes of cyclic groups whose cardinality has some special arithmetic properties.

\begin{corollary}\label{c:Bose} For any prime power $q\ne 3$ the cyclic group $C_{q^2-1}$ has difference size
$$\Delta[C_{q^2-1}]\le q-1+\frac{12}{\sqrt{73}}\sqrt{q-1}<q-1+\sqrt{2q-2}.$$
\end{corollary}

\begin{proof} For $q\in \{2,4,5,7,8,9\}$ the upper bound
$$\Delta[C_{q^2-1}]\le q-1+\frac{12}{\sqrt{73}}\sqrt{q-1}$$ can be verified using known values of the difference sizes $\Delta[C_{q^2-1}]$ given in Table~\ref{tab:cycl}.

So, assume that $q>9$ and hence $q\ge 11$. In this case Theorem~\ref{t:ub-Cn} guarantees that $\Delta[C_{q-1}]\le\frac{12}{\sqrt{73}}\sqrt{q-1}$.
Applying Theorem~\ref{t:Bose}, we conclude that the group $C_{q^2-1}$ has difference size
$$\Delta[C_{q_2-1}]\le q-1+\Delta[C_{q-1}]\le q-1+\frac{12}{\sqrt{3}}\sqrt{q-1}.$$
\end{proof}

By analogy we can derive the following upper bound for $\Delta[C_{p^2-p}]$ from Theorem~\ref{t:Rusza}.

\begin{corollary}\label{c:Rusza}
For any prime number $p$ the cyclic group $C_{p^2-p}$ has difference size
$$\Delta[C_{p^2-p}]\le p-3+\tfrac{12}{\sqrt{73}}\big(\sqrt{p}+\sqrt{p-1}\big).$$
\end{corollary}

In Table~\ref{tab:cycl} we present the results of computer calculation
of the difference sizes of cyclic groups of order $\le 100$. In this
table  $$
lb[C_n]:=\left\lceil\tfrac{1+\sqrt{4n-3}}2\right\rceil
$$ is the lower bound given in Proposition~\ref{p:BGN}(1) and
$$
ub[C_n]:=\begin{cases}q+1&\mbox{if $n=q^2+q+1$ for some prime power $q$},\\
q-1+\Delta[C_{q-1}]&\mbox{if $n=q^2-1$ for some prime power $q$,}\\
p-3+\Delta[C_p]+\Delta[C_{p-1}]&\mbox{if $n=p^2-p$ for some prime number $p$,}
\end{cases}
$$
is the upper bound for $\Delta[C_n]$ given in Theorems~\ref{t:Singer}, \ref{t:Bose} and \ref{t:Rusza}.  With the boldface font we denote the numbers
$n\in\{7,13,21,31,57,73,91\}$, equal to $1+q+q^2$ for a prime power $q$. For such numbers we know the exact value $\Delta[C_n]=lb[C_n]=1+q$.

\begin{table}[ht]
\caption{Difference sizes of cyclic groups $C_n$ for $n\le100$}\label{tab:cycl}
\begin{tabular}{|c|c|c|c|c|c|c|c|c|c|c|c|c|c|c|c|c|c|c|c|c|}
\hline
$n$   & 1 & {2} &     {3} & 4 & 5 & {6} & {\bf7} & {8} & 9 &10 &11 &12 &{\bf
13} &14 &{15} &16 &17 &18 &19 & {20}  \\
\hline
$lb[C_n]$ & 1 & 2 &     2 & 3 & 3 & 3 & 3 & 4 & 4 & 4 & 4 & 4 & 4 & 5 &
5 & 5 & 5 & 5 & 5 & 5 \\
$\Delta[C_n]$ & 1 & 2 &     2 & 3 & 3 & 3 & {3} & 4 & 4 & 4 & 4 & 4 & {4} &
5 & 5 & 5 & 5 & 5 & 5 & 6 \\
$ub[C_n]$   &  & 2 &     2 &  &  & 4 & 3 & 4 &  & & & &4 & &5 & & & & & 8  \\
$\Delta\big[\lceil\frac{n-1}2\rceil\big]$ & 1 & 2 & 2 & 3 & 3 & 3 &3 &4 &4 & 4 & 4& 4 &4 &5 &5 &5 &5 &5
& 5 & 6\\
\hline
\hline
$|C_n|$ & {\bf 21} & 22& 23& {24}& 25& 26& 27& 28& 29& 30& {\bf 31}& 32&
33& 34& 35& 36& 37& 38& 39& 40  \\
\hline
$lb[C_n]$& 5 & 6 & 6 & 6 & 6 & 6 & 6 & 6 & 6 & 6 & 6 & 7 & 7 & 7 & 7 & 7
& 7 & 7 & 7 & 7 \\
$\Delta[C_n]$& {5} & 6 & 6 & 6 & 6 & 6 & 6 & 6 & 7 & 7 & {6} & 7 & 7 & 7 & 7
& 7 & 7 & 8 & 7 & 8 \\
$ub[C_n]$ & 5 & & & 7& & & & & & & 6& &
& & & & & & &   \\
$\Delta\big[\lceil\frac{n-1}2\rceil\big]$& 6& 6& 6& 6& 6& 6& 6& 7& 7& 7& 7& 7& 7& 7& 7& 7& 7& 8& 8& 8\\
\hline
\hline	
$|C_n|$  & 41 & {42}& 43& 44& 45& 46& 47& {48}& 49& 50& 51& 52& 53& 54& 55&
56& {\bf57}& 58& 59& 60  \\
\hline
$lb[C_n]$ & 7 & 7 & 7 & 8 & 8 & 8 & 8 & 8 & 8 & 8 & 8 & 8 & 8 & 8 & 8 &
8 & 8 & 9 & 9 & 9 \\
$\Delta[C_n]$ & 8 & 8 & 8 & 8 & 8 & 8 & 8 & 8 & 8 & 8 & 8 & 9 & 9 & 9 &
9 & 9 & {8} & 9 & 9 & 9 \\
$ub[C_n]$  &  & 10& & & & & & 9& & & & & & & & & 8& & &   \\
$\Delta\big[\lceil\frac{n-1}2\rceil\big]$ &8 & 8& 8& 8& 8& 8& 8& 8& 8& 9& 9& 9& 9& 9& 9& 9& 9& 9& 9&
10\\
\hline
\hline	
$|C_n|$ & 61 & 62& {63}& 64& 65& 66& 67& 68& 69& 70& 71& 72 & {\bf73} & 74
& 75 & 76 & 77 & 78 & 79 & {80}  \\
\hline
$lb[C_n]$& 9 & 9 & 9 & 9 & 9 & 9 & 9 & 9 & 9 & 9 & 9 & 9   & 9 & 10 & 10
& 10 & 10 & 10 & 10 & 10 \\
$\Delta[C_n]$& 9 & 9 & 9 & 9 & 9 & 10 & 10 & 10 & 10 & 10 & 10 & 10 & {9}
& 10 & 10 & 10 & 10 & 10 & 10 & 11 \\
$ub[C_n]$ &  & & 10& & & & & & & & & & 9 & & & & & & & 12 \\
$\Delta\big[\lceil\frac{n-1}2\rceil\big]$&10 & 10& 10& 10&10 & 10 & 10 & 10 & 10 & 10 & 10 & 10 & 10& 10
& 10 & 11 & 11 & 11 & 11 & 11\\
\hline
\hline	
$|C_n|$ & 81 & 82 & 83 & 84 & 85 & 86 & 87 & 88 & 89 & 90 & {\bf91} & 92
& 93 & 94 & 95 & 96 & 97 & 98 & 99 & 100  \\
\hline
$lb[C_n]$& 10& 10 & 10 & 10 & 10 & 10 & 10 & 10 & 10 & 10 & 10 & 11 & 11
& 11 & 11 & 11 & 11 & 11 & 11 & 11 \\
$\Delta[C_n]$& 11& 11 & 11 & 11 & 11 & 11 & 11 & 11 & 11 & 11 & {10} & 11
& 12 & 12 & 12 & 12 & 12 & 12 & 12 & 12 \\
$ub[C_n]$ & & & & & & & & & & & 10 & & & & & & & & &  \\
$\Delta\big[\lceil\frac{n-1}2\rceil\big]$&11 & 11& 11& 11&11 & 11 & 11 & 11 & 11 & 11 & 11 & 12 & 12& 12
& 13 & 13 & 13 & 13 & 13 & 13\\
\hline
\end{tabular}
\end{table}

\begin{remark} For $n=20$ we get the strict inequality
$\Delta[C_n]>lb[C_n]$, which answers Problem~\ref{prob1} in negative.
\end{remark}

The results of computer calculations suggest the following questions.

\begin{question} Is $|\Delta[C_{n+1}]-\Delta[C_n]|\le1$ for every
$n\in\IN$?
\end{question}

\begin{question} Is $\sup_{n\in\IN}(\Delta[C_n]-lb[C_n])<\infty$?
\end{question}

The following problem seems to be the most intriguing (see {\tt
http://mathoverflow.net/questions/262317}).

\begin{problem}\label{prob:lim1} Is $\Delta[C_n]=(1+o(1))\cdot\sqrt{n}$ as
$n\to\infty$? Equivalently, is  $\lim_{n\to\infty}\eth[C_n]=1?$
\end{problem}

Theorem~\ref{t:Singer}, Corollaries~\ref{c:Bose}, \ref{c:Rusza}, and Proposition~\ref{p:BGN}(2) allows us to produce many subsequences of the sequence $(\eth[C_n])_{n=1}^\infty$ tending to the unit. In particular,
$$\lim_{q\to\infty}\eth[C_{q^2+q+1}]=\lim_{q\to\infty}\eth[C_{q^2-1}]
=\lim_{p\to\infty}\eth[C_{p^2-p}]=1,
$$where $q$ runs over prime powers and $p$ runs over prime numbers to infinity. On the other hand, we do not know the answers to the following problems (which are weaker versions of Problem~\ref{prob:lim1}).

\begin{problem} Is $\lim_{p\to\infty}\eth[C_p]=1$?
\end{problem}

\begin{problem} Let $p$ be a prime number. Is $\lim_{k\to\infty}\eth[C_{p^k}]=1$?
\end{problem}

\begin{problem} Is $\limsup_{n\to\infty}\eth[C_n]=\limsup_{p\to\infty}\eth[C_p]$?
\end{problem}

\section{Acknowledgment}

The authors would like to express their sincere thanks to Oleg Verbitsky
who turned their attention to the theory of difference sets and their
relation with difference bases, to Alex Ravsky for valuable discussions
on perfect rulers, to MathOverflow users Lucia, Seva, and Sean Eberhard for valuable comments to the questions asked by the first author on MathOverflow.


\begin{thebibliography}{9}



\bibitem{And} D.~Andrica, {\em Note on a conjecture in prime number theory},
    Studia Univ. Babes--Bolyai Math. {\bf 31}:4 (1986) 44--48.

\bibitem{BHP} R.C.~Baker, G.~Harman, J.~Pintz, {\em The difference between consecutive primes, II}. Proc. London Math. Soc. {\bf 83}:3 (2001), 532--562.


\bibitem{BGN} T.~Banakh, V.~Gavrylkiv, O.~Nykyforchyn, {\em Algebra in
superextension of groups, I:
zeros and commutativity}, Algebra Discr. Math. {\bf 3} (2008), 1--29.

\bibitem{BG} T.~Banakh, V.~Gavrylkiv, {\em Algebra in the
superextensions of twinic groups}, Dissertationes Math. {\bf 473}
(2010), 74~pp.

\bibitem{BH} E.~Bertram, M.~Herzog, {\em Bounds on character degrees and class
    numbers of finite nonabelian simple groups}, Groups--St. Andrews 1989, Vol.
    1, 46--51, London Math. Soc. Lecture Note Ser., 159, Cambridge Univ. Press,
    Cambridge, 1991.

\bibitem{Bose} R.C.~Bose, {\em An affine analogue of Singer�s theorem}, J. Indian Math. Soc. {\bf 6} (1942) 1--15.

\bibitem{Chowla} R.C.~Bose, S.~Chowla, {\em Theorems in the additive theory of numbers}, Comment. Math. Helvetici {\bf 37} (1962-63) 141�-147.

\bibitem{Dusart} P.~Dusart, {\em Estimates of $\psi,\theta$ for large values of $x$ without the Riemann hypothesis}, Math. Comp. {\bf85}:298 (2016) 875--888.

\bibitem{EG} P.~Erd\H os, I.~G\'al, {\em On the representation of $1,
2,\dots, N$ by differences}, Nederl. Akad. Wetensch., Proc. {\bf51},
(1948) 1155--1158.

\bibitem{FKL} L.~Finkelstein, D.~Kleitman, T.~Leighton, {\em Applying
the classification theorem for finite simple groups to minimize pin
count in uniform permutation architectures}, VLSI algorithms and
architectures (Corfu, 1988), 247--256, Lecture Notes in Comput.
Sci., 319, Springer, New York, 1988.

\bibitem{Ghory} I.~Ghory, {\em
 Prime Numbers: The Most Mysterious Figures in Math}, John Wiley \&\ Sons, Inc., 2005, p. 13.

\bibitem{Golay} M.~Golay, {\em Notes on the representation of $1,\,2,\,\ldots ,\,n$ by differences}, J. London Math. Soc. (2) {\bf 4} (1972) 729--734.

\bibitem{G} D.~Gordon, {\em The prime power conjecture is true for
$n<2\,000\,000$}, Electron. J. Combin. {\bf 1} (1994), Research
Paper 6,


\bibitem{Leech} J.~Leech, {\em On the representation of $1,2,\dots,n$
by differences}, J. London Math. Soc. {\bf31} (1956), 160--169.

\bibitem{KL} G.~Kozma, A.~Lev, {\em Bases and decomposition numbers of
finite groups}, Arch. Math. (Basel) {\bf 58}:5 (1992), 417--424.

\bibitem{MP} E.~Moore, H.~Pollatsek, {\em Difference sets. Connecting
algebra, combinatorics, and geometry}, Amer. Math. Soc., Providence,
RI, 2013.



\bibitem{RR} L.~R\'edei, A.~R\'enyi, {\em On the representation of the
numbers $1,2,\dots, N$ by means of differences}, Mat. Sbornik N.S.
{\bf 24}(66) (1949), 385--389.


\bibitem{Rusza} I.Z.~Ruzsa, {\em Solving a linear equation in a set of integers I}, Acta
Arithmetica LXV. {\bf 3} (1993) 259--282.

\bibitem{Singer} J.~Singer, {\em A theorem in finite projective geometry
and some applications to number theory}, Trans. Amer. Math. Soc.
{\bf 43}:3 (1938), 377--385.



\bibitem{Wiki} Wikipedia, {\em Sparse ruler}, ({\tt https://en.wikipedia.org/wiki/Sparse$\backslash$ruler}).

\end{thebibliography}
\end{document}